\documentclass[10pt]{article}

\usepackage{latexsym}
\usepackage{amsmath}
\usepackage{amsthm}
\usepackage{amssymb}
\usepackage{enumerate}
\usepackage{url}
\usepackage{xcolor}

\title{Interpreting the action of the endomorphism monoid of the rationals}

\author{John K. Truss, University of Leeds,}
\date{Edith Vargas-Garc\'\i a, ITAM.}

\begin{document}
\maketitle 
\newtheorem{lemma}{Lemma}[section]
\newtheorem{theorem}[lemma]{Theorem}
\newtheorem{proposition}[lemma]{Proposition}
\newtheorem{corollary}[lemma]{Corollary}
\newtheorem{definition}[lemma]{Definition}
\newtheorem{remark}[lemma]{Remark}
\setcounter{footnote}{1}\footnotetext{Department of Pure Mathematics, University of Leeds, Leeds LS2 9JT, UK, e-mail pmtjkt@leeds.ac.uk, and
Department of Mathematics, ITAM Río Hondo 1,Ciudad de México 01080, México, e-mail:edith.vargas@itam.mx.}
\setcounter{footnote}{2}\footnotetext{The second author gratefully acknowledges support by Asociación Mexicana de Cultura, A.C.}

\newcounter{number}

\begin{abstract}  In this paper, we define the action of $M$, the monoid of embeddings of $({\mathbb Q}, \le)$, on ${\mathbb Q}$, in the monoid $(M, \circ)$. That is we show that ${\mathbb Q}$ itself can be 
interpreted in $(M, \circ)$, and in addition, so can the action of $M$ on $\mathbb Q$. This is extended to the monoid $E$ of all endomorphisms of $({\mathbb Q}, \le)$.   \end{abstract}

2010 Mathematics Subject Classification 08A35

keywords: rationals, first order interpretation, embedding, endomorphism.

\section{Introduction}
Our goal here is to show directly how one can (first order) interpret the action of the monoid $M$ of embeddings of $({\mathbb Q}, \le)$ on $\mathbb Q$ inside $(M, \circ)$. In other words, we represent 
the points of the domain, here the set of rational numbers, as the equivalence classes of certain monoid elements under a definable equivalence relation, and in addition, the action of the monoid on the 
domain presented via this interpretation. The methods used are standard in the theory of ordered permutation groups, see \cite{Glass}, but need to be extended to apply to the present context. Using tricks 
from the paper \cite{Edith}, we obtain the analogous result for the monoid $E$ of endomorphisms of $({\mathbb Q}, \le)$.

\section{Defining some Formulae}

This section contains the technical material which enables us to interpret first of all the domain $\mathbb Q$ in the group $G =$Aut$({\mathbb Q}, \le)$, and then from this also the action of $G$ on this
interpretation. In the next section, we argue about the monoids $M$ and $E$ of self-embeddings of $({\mathbb Q}, \le)$, and of all endomorphisms of $({\mathbb Q}, \le)$, respectively, and show that the 
interpretation of the action of $G$ can be extended to these larger monoids. Since $G$ is definable in both $M$ and $E$, as the family of invertible elements, the interpretation in $G$ can be recast as one 
in $M$ or $E$. In this section all variables range over $G$. 

We look at various formulae which are meant to characterize certain special elements of $G$ (usually `up to conjugacy'). Since our actions are on the left, we write $x^y$ for the conjugate $yxy^{-1}$ rather 
than $y^{-1}xy$. We write $\mbox{conj}(x,y)$ for $\exists z(x^z = y)$, saying that two elements are conjugate. The {\em support} of an element is the set of points that it moves. As remarked above, much of 
this material already exists in \cite{Glass}. Proofs omitted can be found in \cite{Truss}.

If $f\in G$, an orbital $X$ of $f$ is a convex subset `spanned' by an orbit. This means that it is an equivalence class of $\mathbb Q$ under the relation $x \sim y$ if for some integers $m$ and $n$, 
$f^m x \le y \le f^n x$. There are $3$ `kinds' of orbital. $X$ has `parity $+1$' if for some (any) $x \in X, x < f(x)$. It has `parity $-1$' if for some (any) $x \in X, x > f(x)$. It has `parity $0$' if for 
any $x \in X, f(x) = x$ (which means that it is then a singleton, fixed by $f$). To give two easy examples:

If $f(x) = x+1$ then $f$ has just one orbital, the whole of ${\mathbb Q}$, and this has parity $+1$.

If $f(x) = 2x$, then $f$ has $3$ orbitals, $(-\infty, 0)$ of parity $-1$, $\{0\}$ of parity $0$, and $(0, \infty)$ of parity $+1$.

The {\em orbital pattern} of $f$ is then the $3$-coloured linear order with colours in $\{\pm 1, 0\}$ comprising the family of all orbitals regarded as a linearly ordered set under the induced ordering, with 
the associated colouring. Charles Holland remarked that $f$ and $g$ are conjugate if and only if their orbital patterns are isomorphic.

The following formula is due to McCleary \cite{McCleary}:

\vspace{.1in}

\noindent $\mbox{comp}(x):~ (\exists y)(\exists z)(y \neq 1 \wedge z \neq 1 \wedge (\forall t)~(y z^{x^t} = z^{x^t} y)).$

\vspace{.1in}

Here `comp' stands for `comparable' (with the identity), and it expresses the fact that the element concerned is either positive or negative, where $f$ is said to be `positive' if $f(x) \ge x$ for all $x$, 
and `negative' if $f(x) \le x$ for all $x$. Note that if two elements of $G$ have disjoint supports, then they commute, and the formula captures the situation that there are non-trivial elements, 
instantiations of $y$ and $z$, such that the support of $y$ is to the left of that of $z$, so that any conjugate of $z$ by a positive element still has support disjoint from that of $y$ (or the other way 
round in the case of a negative element). 

\begin{lemma}\label{2.1} For $f \in G$, $G \models {\rm comp}(f)$ if and only if $f$ is either positive or negative. \end{lemma}

Next we move towards expressing disjointness of supports.

\vspace{.1in}

\noindent $\mbox{apart}(x,y):~ (\exists z)(\mbox{comp}(z) \wedge z \neq 1 \wedge (\forall t)~(x y^{z^t} = y^{z^t} x)).$

\begin{lemma}\label{2.2} For any $f, g \in G$, $G \models {\rm apart}(f, g)$ if and only if either ${\rm supp}(f) < {\rm supp}(g)$ or ${\rm supp}(g) < {\rm supp}(f)$.  \end{lemma}

Note that the degenerate case in which $x$ or $y$ is 1 is allowed. The proof of this lemma follows similar lines to the preceding one with the role of the variables permuted.

The next formula characterizes `bumps', which are non-identity elements having exactly one non-trivial orbital.

\vspace{.1in}

\noindent $\mbox{bump}(x): x \neq 1 \wedge (\forall y)(\forall z)(x = yz \wedge \mbox{apart}(y,z) \rightarrow (y = 1 \vee z = 1)).$

\begin{lemma}\label{2.3} For any $f \in G$, $G \models {\rm bump}(f)$ if and only if $f$ is a bump.  \end{lemma}

The intuition is that $x$ is `indecomposable' into pieces with disjoint supports. 

In the next lemma we need the notion of `restriction'. If $f, g \in G$ then $f$ is a {\em restriction} of $g$ if for all $x$ in the support of $f$, $f(x) = g(x)$, and if $X$ is the support of $f$, then $f$
is said to be the {\em restriction} of $g$ to $X$.

\vspace{.1in}

\noindent $\mbox{orbital}(x,y): \mbox{bump}(x) \wedge (\exists y_1)(\exists y_2)(\mbox{apart}(x, y_1) \wedge \mbox{apart}(x, y_2) \wedge \mbox{apart}(y_1, y_2) \wedge y = x y_1 y_2).$

(Note that in this formula, one or both of $y_1$, $y_2$ could be the identity.) 

\begin{lemma}\label{2.4} For any $f \in G$, $G \models {\rm orbital}(f,g)$ if and only if $f$ is the restriction of $g$ to one of its non-trivial orbitals.  \end{lemma}

\noindent ${\rm disj}(x, y): (\forall z)(\forall t)({\rm orbital}(z,x) \wedge \mbox{orbital}(t,y) \to \mbox{apart}(z,t)).$

\noindent ${\rm cont}(x, y): (\forall z)({\rm disj}(y, z) \to {\rm disj}(x,z))$.

\noindent ${\rm samesupport}(x, y): {\rm cont}(x, y) \wedge {\rm cont}(y,x)$.

\begin{lemma}\label{2.5} For any $f, g \in G$, 

(i) $G \models {\rm disj}(f,g)$ if and only if $f$ and $g$ have disjoint supports.  

(ii) $G \models {\rm cont}(f,g)$ if and only if the support of $f$ is contained in that of $g$.

(ii) $G \models {\rm samesupport}(f,g)$ if and only if $f$ and $g$ have equal supports. \end{lemma}

The next two formulae give us some flexibility in dealing with bumps and their supports.

\vspace{.1in}

\noindent $\mbox{between}(x,y,z): {\rm bump}(x) \wedge {\rm bump}(y) \wedge {\rm bump}(z) \wedge {\rm disj}(x,y) \wedge {\rm disj}(x,z) \wedge {\rm disj}(y,z) \wedge \forall t({\rm disj}(t,y) \to 
{\rm disj}(x^t, z)).$

\begin{lemma}\label{2.6} For any $f, g, h \in G$, $G \models {\rm between}(f,g,h)$ if and only if $f, g, h$ are bumps with pairwise disjoint supports such that the support of $g$ lies between the other two 
supports.  \end{lemma}

\vspace{.1in}

\noindent $\mbox{adj}(x,y): \mbox{bump}(x) \wedge \mbox{bump}(y) \wedge \neg \exists z \, \mbox{between}(x, z, y).$

\noindent $\mbox{union}(x,y,z): \mbox{bump}(z) \wedge \forall t(\mbox{bump}(t) \wedge \mbox{cont}(x,t) \wedge \mbox{cont}(y,t) \to \mbox{cont}(z,t)).$

\begin{lemma}\label{2.7} For any $f, g \in G$, $G \models {\rm adj}(f,g)$ if and only if $f, g$ are bumps with disjoint supports which have an endpoint in common, and for any $f$, $g$ and $h$, 
$G \models {\rm union}(f,g,h)$ if and only if the support of $h$ is the least convex open interval containing the supports of $f$ and $g$. \end{lemma}

We can also now express that a bump has support the whole of ${\mathbb Q}$, or is bounded above or below:

\vspace{.1in}

\noindent ${\rm coterm}(x): {\rm bump}(x) \wedge (\forall y)({\rm disj}(x,y) \to y = 1).$

\noindent ${\rm boundedbump}(x): {\rm bump}(x) \wedge (\exists y)({\rm comp}(y) \wedge {\rm disj}(x,x^y) \wedge {\rm disj}(x,x^{y^{-1}})).$

\noindent ${\rm cof}(x): {\rm bump}(x) \wedge \neg {\rm coterm}(x) \wedge \neg {\rm boundedbump}(x).$

\begin{lemma}\label{2.8} For any $f \in G$, 

(i) $G \models {\rm coterm}(f)$ if and only if $f$ has just one orbital, which is the whole of $\mathbb Q$,

(ii) $G \models {\rm boundedbump}(f)$ if and only if $f$ has just one non-trivial orbital, which is bounded (above and below),

(iii) $G \models {\rm cof}(f)$ if and only if $f$ has just one orbital, which has the form $(-\infty, a)$ or $(a, \infty)$ for some $a$ (which may be rational or irrational). \end{lemma}

So we now concentrate on cofinal elements---ones satisfying `cof' (though coterminal elements are also needed from time to time).

\vspace{.1in}

\noindent ${\rm oppsupport}(x,y):  {\rm cof}(x) \wedge {\rm cof}(y) \wedge {\rm disj}(x,y) \wedge (\forall z)({\rm disj}(x, z) \wedge {\rm disj}(y,z) \to z = 1).$

\begin{lemma}\label{2.9} For any $f, g \in G$, $G \models {\rm oppsupport}(f,g)$ if and only if for some $a$ (which may be rational or irrational), $f$ and $g$ are cofinal elements one of which has
support $(-\infty, a)$ and the other has support $(a, \infty)$.  \end{lemma}
From these we get 

\vspace{.1in}

\noindent ${\rm codesame}(x,y): \mbox{cof}(x) \wedge \mbox{cof}(y) \wedge {\rm samesupport}(x,y) \vee \mbox{oppsupport}(x,y).$ 

\begin{lemma}\label{2.10} For any $f, g \in G$, $G \models {\rm codesame}(f,g)$ if and only if for some $a$, $f$ and $g$ are cofinal elements with supports $(-\infty, a)$ or $(a, \infty)$.  \end{lemma}

Thus we have expressed when two cofinal elements `encode' the same point. It remains to express when this point is rational or irrational.

\vspace{.1in}

\noindent $\mbox{gauge}(x,y): \mbox{coterm}(x) \wedge \mbox{coterm}(y) \wedge (\exists z)(x^z = y) \wedge xy = yx$

$\hspace{1in} \wedge (\forall z)(\forall t)(xz = zx \wedge xt = tx \wedge yz = zy \wedge yt = ty \to zt = tz).$

\vspace{.1in}

This expresses that $x$ and $y$ are coterminal elements of the same parity (not needed, but makes things easier to understand), and they generate an abelian group $H = \langle x, y \rangle$, and the 
centralizer of $H$ is abelian. The reason for the choice of the word `gauge' is that such a pair $(f,g)$ provides some sort of `measurement' or scaling of the rational line. In fact for what follows, it is 
convenient to work with the induced action of members of $G$ on the real numbers. Indeed, we may view $G$ as a subgroup of Aut$({\mathbb R}, <)$, since any member $f$ of $G$ induces a member of 
Aut$({\mathbb R},<)$ by continuity ($f(x) = \mbox{ sup }{f(q):~ q \in {\mathbb Q},~ q \le x}$).

\begin{lemma}\label{2.11} For any $f, g \in G$, $G \models {\rm gauge}(f,g)$ if and only if for some embedding $\theta$ from $\mathbb Q$ to a dense subset $X$ of $\mathbb R$, and irrational number $\alpha$,
$\theta f \theta^{-1}(x) = x+1$, $\theta g \theta^{-1}(x) = x + \alpha$, for all $x \in X$.  \end{lemma}

\begin{proof} If $f$ and $g$ are as stated, then as $\theta f \theta^{-1}$ and $\theta g \theta^{-1}$ are both translations, they commute, and hence so do $f$ and $g$. Let $h$ commute with both $f$ and $g$. 
Then $\theta h\theta^{-1}(x + n) = \theta h \theta^{-1}\theta f^n \theta^{-1}(x) = \theta hf^n\theta^{-1}(x) = \theta f^n h \theta^{-1}(x) = \theta h \theta^{-1} + n$, and similarly 
$\theta h \theta^{-1}(x+n\alpha) = \theta h \theta^{-1}(x) + n\alpha$. Hence $\theta h \theta^{-1}(m + n\alpha) = \theta h \theta^{-1}(0) + m + n\alpha$, so the set of real numbers $x$ such that
$\theta h \theta^{-1}(x) - x = \theta h \theta^{-1}(0)$ contains all of the form $m + n\alpha$. As $\alpha$ is irrational, this is a dense set, and as $\theta h \theta^{-1}(x) - x$ is continuous, it is 
contains all of ${\mathbb R}$. Hence $\theta h \theta^{-1}(x) = \theta h\theta^{-1}(0) + x$ for all $x \in X$. Thus $\theta h\theta^{-1}$ is a translation (by a member of $X$). 
So any two elements commuting with both $\theta f\theta^{-1}$ and $\theta g\theta^{-1}$ are translations, and so themselves commute. Hence any two members of $G$ commuting with both $f$ and $g$ commute.
This establishes what is wanted for gauge$(f,g)$.

Conversely, suppose that gauge$(f,g)$ holds. Then by replacing $f$ and $g$ by their inverses if necessary, we assume that they are positive. The main point is to show that every orbit of the group 
$H = \langle f, g \rangle$ generated by $f$ and $g$ is dense. Without loss of generality, $f$ is translation by $+1$. As remarked above, we work in ${\mathbb R}$ where necessary (and then restrict at the 
end of the argument to a countable dense subset). Suppose for a contradiction that some orbit $Y$ on ${\mathbb R}$ is not dense, and let $I$ be a maximal open interval of its complement. Consider the family 
${\mathcal I} = \{f^mg^nI: m, n \in {\mathbb Z}\}$, all of whose elements are also maximal open intervals of the complement of $Y$. Let $J = \{(m,n): f^mg^nI = I\}$. Since $f$ and $g$ commute, this is a 
subgroup of ${\mathbb Z}^2$, and since $I$ is not fixed by either $f$ or $g$, $(0,0)$ is the only member of $J$ having a zero co-ordinate. We now consider the least positive integer $m$ (if any) such that 
some $(m,n) \in J$. This has both co-ordinates non-zero, and it readily follows that $J$ is cyclic and is generated by $(m,n)$. Thus for some fixed $(m,n)$ (which may now also be $(0,0)$), $J$ consists of 
all multiples of $(m,n)$. Let $h_1$ and $h_2$ be non-commuting order-automorphisms of $I$ that each commute with $f^mg^n$ on $I$. (Note here that $f^mg^n$ is now a single fixed element of Aut$(I,<)$, so it 
is standard to find non-commuting members of its centralizer). Then by copying the action of $h_1, h_2$ on $I$ to all member $f^rg^sI$ of ${\mathcal I}$ (using $f^rg^s h_i (f^rg^s)^{-1}$ on $f^rg^sI$) and 
fixing all other points, we find still non-commuting order-automorphisms of ${\mathbb R}$ that both commute with $f$ and $g$, contrary to gauge$(f,g)$. The conclusion is that each orbit is dense. To 
conclude this argument, we appeal to some classical results. Since $f$ is translation by $1$, and $g$ commutes with it, they can be viewed as acting on the unit circle $C$ in ${\mathbb C}$ via the map 
$x \mapsto e^{2\pi ix}$, and then $f$ is the identity. More precisely, let ${\tilde f}$ and ${\tilde g}$ be defined on $C$ by letting ${\tilde f}(e^{2 \pi i x}) = e^{2 \pi i f(x)}$, and similarly for 
${\tilde g}$. Because $f$ is periodic with period $1$, ${\tilde f}$ is well-defined, and is actually the identity. But ${\tilde g}$ is also well-defined. For if $e^{2 \pi i x} = e^{2 \pi i y}$, then for 
some integer $n$, $y = x+n$, and hence $g(y) = g(x+n)$, and hence $e^{2 \pi i g(x)}  = e^{2 \pi i g(x+n)}  = e^{2 \pi i g(y)}$, which says that ${\tilde g}$ is well-defined. Since all orbits of 
$\langle f, g \rangle$ on $\mathbb R$ are  dense, all orbits of ${\tilde g}$ on the unit circle are dense. By the results given on pages 32-41 of \cite{Nitecki}, it then follows that the action of $g$ on 
the circle is conjugate to a rotation through $2\pi$ times an irrational angle which we also write as $\alpha$ (this is called the `rotation number' of the map, which in some sense is its `average' 
rotation). Lifting this back to $\mathbb R$, it follows that by replacing $g$ by a conjugate using a conjugacy that commutes with $f$, we may suppose that it is translation by an irrational $\alpha$.
\end{proof}
 
We can now distinguish rational and irrational.

\vspace{.1in}

\noindent $\mbox{rational}(x) : (\exists y)(\exists z)({\rm gauge}(y, z) \wedge$

\hspace{1in} $((\forall t)((\exists u)(t = x^u) \to (\exists u)(yu = uy \wedge zu = uz \wedge \mbox{codesame}(t, x^u))).$

\begin{lemma}\label{2.12} If $f \in G$, $G \models {\rm rational}(f)$ if and only if $f$ is a cofinal element having support $(-\infty, q)$ or $(q, \infty)$ for some rational $q$. \end{lemma}
 
\begin{proof} First suppose that the formula holds, and let $g$ and $h$ be witnesses for $y$ and $z$. By Lemmas \ref{2.8}(iii) and \ref{2.11}, we may view $G$ as acting on some countable dense linear order 
$X$ without endpoints, so that and $g$ and $h$ are translations by 1 and $\alpha$ respectively, for some irrational $\alpha$, and $f$ has support $(-\infty, q)$ or $(q, \infty)$ for some real $q$, suppose 
$(q,\infty)$ for ease. The members of $G$ which are conjugate to $f$ are precisely the cofinal elements of the same parity having support $(r, \infty)$ where $q \in X \leftrightarrow r \in X$. Let $k$ be 
any value for $t$ in the formula. Thus $f$ and $k$ are conjugate, so $k$ is a cofinal element having support $(r, \infty)$ for some $r$, of the same parity as $f$ and $q \in X \leftrightarrow r \in X$. Here 
$r$ may be taken as any point such that $q \in X \leftrightarrow r \in X$. Now as shown above, if $l$ is a witness for $u$ on its second occurrence, then $l$ must be a translation by a member of $X$, and 
since $k$ and $f^l$ code the same point, it must take $q$ to $r$. Since $X$ is countable, there can be only countably many such $l$, and hence $q$ must lie in $X$ (as otherwise there would be uncountably 
many values of $r$ available). When $X$ is `reidentified' with ${\mathbb Q}$, this tells us that $q$ is rational.

Conversely, suppose that $f$ is a cof element having support $(q, \infty)$ where $q$ is rational. By the above, there is a pair $(g, h)$ satisfying gauge, and by Lemma \ref{2.11}, relabelling ${\mathbb Q}$, 
we may regard it as $X$, where $q \in X$ and for some irrational $\alpha, g(x) = x+1$ and $h(x) = x + \alpha$. Now any element $k$ of $G$ which is conjugate to $f$ is a cofinal element having support 
$(r, \infty)$ for some $r \in X$. There is a translation of $X$ which takes $q$ to $r$, and this conjugates $f$ to some cofinal element having support $(r, \infty)$, which codes the same point as $k$. 
\end{proof}

\begin{theorem}\label{2.13} The set $\mathbb Q$, and action of $G$, the automorphism group of $({\mathbb Q}, \le)$, on $\mathbb Q$, can be interpreted in $(G, \circ)$. \end{theorem}
 
\begin{proof} We may interpret rational numbers as members of $G$ which fulfil the formula `rational', two of which are identified if they fulfil the formula `codesame'. We now let $G$ act on 
equivalence classes of rational cofinal elements under `codesame' by conjugacy, which gives rise to the desired formula act$(x,y,z)$ which says that $y$ and $z$ are rational cofinal elements, and 
codesame($y^x,z)$:

\vspace{.1in}

\noindent $\mbox{act}(x, y, z): \mbox{rational}(y) \wedge \mbox{rational} (z) \wedge \mbox{codesame}(y^x, z).$

\vspace{.1in}

To see that this is correct, suppose that $G \models \mbox{act}(f, g, h)$. By Lemma \ref{2.12}, $g$ and $h$ are cofinal elements of $G$ such that for some rationals $q$ and $r$, $g$ has support 
$(-\infty, q)$ or $(q, \infty)$ and $h$ has support $(-\infty, r)$ or $(r, \infty)$. Since $G \models \mbox{codesame}(g^f, h)$, $g^f$ and $h$ have equal or opposite supports. Since 
supp($g^f) = f(\mbox{supp}(g))$, $f(q) = r$. Conversely, if $f(q) = r$ we can choose cofinal elements $g$ and $h$ having supports $(q, \infty)$ and $(r, \infty)$ respectively, and then $g^f$ has support 
$(r, \infty)$, so $G \models \mbox{codesame}(g^f, h)$.
\end{proof}

We conclude this section by remarking that we can also recover the betweenness relation on this interpretation of $\mathbb Q$. It isn't possible to recover the ordering, since 
${\rm Aut}({\mathbb Q}, <) \cong {\rm Aut}({\mathbb Q}, >)$. The next best thing is to recover `linear betweenness', which is the ternary relation defined by $B(x, y, z)$ if $x \le y \le z$ or 
$z \le y \le x$. Using the formula `act' introduced above, the formula `between'; given as follows:

\vspace{.1in}

\noindent $\mbox{between}(x, y, z): x = y \vee y = z \vee (x \neq y \wedge y \neq z \wedge (\exists t)(\mbox{act}(t, x, y) \wedge \mbox{act}(t, y, z))$

\vspace{.1in}

\noindent precisely expresses linear betweenness on $\mathbb Q$. More precisely, $G \models \mbox{between}(g, h, k)$ if and only if $g, h, k$ encode rationals $q, r, s$ respectively such that 
$q \le r \le s$ or $s \le r \le q$. For instance, if $q < r < s$, there is a positive $l$ witnessing $t$ in the formula which takes $q$ to $r$ and $r$ to $s$, and if $q > r > s$, $l$ can be taken to be 
negative. Conversely, if $l$ exists which takes $q$ to $r$ and $r$ to $s$, then as it is order-preserving, $q < r \Leftrightarrow r < s$.

\section{Proofs of the main results}

In this section we prove the results promised in the introduction, beginning with two small but significant technical points concerning what we are able to express in $M$. 

\begin{lemma}\label{3.1} For any $f \in M$ and $g \in G$, $M \models gf = f$ if and only if the image of $f$ is disjoint from the support of $g$. \end{lemma}

\begin{proof} First suppose that $gf = f$. If $x$ lies in the image of $f$, write $x = fy$. Then $x = fy = gfy = gx$ and so $x$ does not lie in the support of $g$. Hence 
${\rm im}(f) \cap {\rm supp}(g) = \emptyset$.

Conversely, if ${\rm im}(f) \cap {\rm supp}(g) = \emptyset$ then for any $x$, $fx \in {\rm im}(f)$ and so is fixed by $g$, so $gf(x) = f(x)$.   \end{proof}

Now consider the following formula:

\vspace{.1in}

\noindent $\mbox{gap}(x, y) : {\rm bump}(y) \wedge yx = x \wedge (\forall z)({\rm bump}(z) \wedge \neg {\rm disj}(y, z) \wedge zx = x \to {\rm cont}(z, y))$.

\vspace{.1in}

We remark that as shown in Lemma 3.2 of \cite{Edith}, each of $M$ and $G$ is a definable subset of $E$; also $G$ is a definable subset of $M$ (as its set of invertible elements). Hence the results of the 
previous section can be immediately lifted to give interpretations in $M$ or $E$ as the case may be, just by relativizing the relevant formulae to $G$. So in this formula, for instance, $x$ is meant to 
range over $M$, whereas $y$ and $z$ range just over $G$, because they are variables occurring in formulae introduced in section 2. So strictly speaking, the formula should say
`$y$ is invertible, and ..., and for all $z$, if $z$ is invertible then ...'. In what follows we take this as read.

\begin{lemma}\label{3.2} For any $f \in M$ and $g \in G$, $M \models {\rm gap}(f, g)$ if and only if $g$ is a bump whose support is a maximal convex subset of ${\mathbb Q} \setminus {\rm im}(f)$. \end{lemma}

\begin{proof} First suppose that $M \models {\rm gap}(f, g)$. Then $gf = f$, so by Lemma \ref{3.1}, ${\rm supp}(g) \subseteq {\mathbb Q} \setminus {\rm im}(f)$. If we let $h$ be a bump whose support is equal
to the maximal convex subset of ${\mathbb Q} \setminus {\rm im}(f)$ containing ${\rm supp}(g)$, then also $hf = f$ (by Lemma \ref{3.1} again), so the supports of $g$ and $h$ must be equal, hence giving
the desired maximality. 

Conversely, if $g$ is a bump whose support is a maximal convex subset of ${\mathbb Q} \setminus {\rm im}(f)$, then the first part of $\mbox{gap}(f, g)$ is satisfied since by Lemma \ref{3.1}, $gf = f$. 
Maximality ensures that the remaining part is also satisfied, since if $h$ is a bump whose support intersects that of $g$, and such that $hf = f$, ${\rm supp}(h)$ is disjoint from the image of $f$ by
Lemma \ref{3.1}, and it must be contained in the same maximal convex subset of ${\mathbb Q} \setminus {\rm im}(f)$ so is equal to supp($g$).
 \end{proof} 

\vspace{.1in}

For the final arguments, we need to recall key definitions from \cite{Edith}. Let $\mathbb{Q}_2$ stand for the `2-coloured rationals', that is, the ordered set of rational numbers, with colours assigned to 
all points, which we describe as `red' and `blue', which each arise densely. We denote by~$\Gamma$ the family of all $f \in M$ such that~$\mathbb Q$ may be written as the disjoint union
$\bigcup \{A_q\colon q \in \mathbb{Q}_2\}$ of convex subsets~$A_q$ of~$\mathbb Q$ such that $q < r \Rightarrow A_q < A_r$, each~$A_q$ is isomorphic to~$\mathbb Q$, and if~$q$ is a red point of 
$\mathbb{Q}_2$ then~ $|A_q \cap \mathop{\rm im}(f)| = 1$, and if~$q$ is a blue point of $\mathbb{Q}_2$ then $A_q \cap \mathop{\rm im}(f) = \emptyset$. The intuition is that the points of the image 
of~$f$ are spread out as much as they possibly can be. We also require similar families written as $\Gamma^-$, $\Gamma^+$, and $\Gamma^\pm$, which are defined similarly from $\{-\infty\} \cup {\mathbb Q_2}$,
${\mathbb Q_2} \cup \{\infty\}$, and ${\mathbb Q_2} \cup \{\pm \infty\}$ respectively, where all infinite points have colour blue. These are needed to deal with members of $M$ whose support is bounded below 
but not above, above but not below, or above and below respectively. We have various versions of the formula `act', introduced in the proof of Theorem \ref{2.13}, to deal with the action of monoid elements 
under different hypotheses, $\mbox{act}_1$, $\mbox{act}_2$, $\mbox{act}_3$, and $\mbox{act}_4$.

\vspace{.1in}

\noindent $\mbox{act}_1(x, y, z) : \mbox{ rational}(y) \wedge \mbox{ rational} (z) \wedge \mbox{ conj}(y, z) \wedge$

$(\exists z')(\mbox{conj}(z, z') \wedge \mbox{cont}(z',z) \wedge xy = z'x \wedge \exists t(\mbox{gap}(x, t) \wedge \mbox{adj}(t,z') \wedge \mbox{union}(t,z',z)))$.

\begin{lemma}\label{3.3} If $M \models {\rm act}_1(f, g, h)$ and $g$, $h$ encode rationals $q$ and $r$ respectively, then $f(q) = r$. Conversely, if 
$f \in \Gamma \cup \Gamma^- \cup \Gamma^+ \cup \Gamma^\pm$ and $f(q) = r$ where $q$ and $r$ are rationals encoded by conjugate $g, h \in G$, then $M \models {\rm act}_1(f, g, h)$. \end{lemma}

\begin{proof} First suppose that $M \models \mbox{act}_1(f, g, h)$. Since ${\rm rational}(g)$ and ${\rm rational}(h)$, $g$, $h$ encode rationals $q$ and $r$ say. Since they are conjugate, for ease we suppose 
that they have supports $(-\infty, q)$ and $(-\infty, r)$ respectively (with a similar argument if their supports are $(q, \infty)$, $(r, \infty)$). Let $h'$ be a witness for $z'$, which is also conjugate 
to $h$, so has support $(-\infty, r')$. Since $M \models \mbox{cont}(h', h)$, $r' \le r$. Now $fg = h'f$. Since $g$ has support $(-\infty, q)$, $g(q) = q$. Hence $h'f(q) = fg(q) = f(q)$, so $f(q)$ does not 
lie in the support of $h'$, and hence $f(q) \ge r'$. Also, if $a < q$, $g(a) \neq a$, and since $f$ is injective, $h'f(a) = fg(a) \neq f(a)$, which tells us that $f(a)$ lies in the support of $h'$, so 
$f(a) < r'$. Let $k$ be a witness for $t$. Since $M \models {\rm adj}(k, h')$, $r'$ equals the infimum of the support of $k$, and as this is non-empty and disjoint from the image of $f$, it follows that 
$r' < f(q)$. Since $M \models {\rm gap}(f, k)$, $(r', f(q))$ equals the support of $k$. Since $M \models {\rm union}(k, h', h)$ we deduce that ${\rm supp}(h) = (-\infty, f(q))$. Hence $f(q) = r$.

Conversely, let $f \in \Gamma$, and rational and conjugate $g, h \in G$ encode $q$ and $r$ respectively. We have to show that $M \models {\rm act}_1(f, g, h)$. Without loss of generality suppose that $g$ 
has support $(-\infty, q)$, from which it follows that $h$ has support $(-\infty, r)$, since they are conjugate. Note that as $f \in \Gamma$, sup$f(-\infty, q) = r' < r$. We shall show how to choose 
suitable witnesses $h'$ for $z'$, and $k$ for $t$. In fact, $k$ can be any bump having support $(r', r)$. To see how to find a suitable $h'$, consider the partial map $fgf^{-1}$ restricted to 
${\rm im}(f) \cap (-\infty, r')$, on which it is well-defined and order-preserving. We observe that 
$fgf^{-1}[{\rm im} f \cap (-\infty, r')] = fg(-\infty, q) = f(-\infty, q) = {\rm im}(f) \cap (-\infty, r')$. By definition of $\Gamma$, and since all points of ${\rm im}(f)$ lie in red intervals, $fgf^{-1}$ 
corresponds to an automorphism of the copy of ${\mathbb Q}_2$ to the left of $r'$, and this automorphism in turn corresponds to an automorphism $h'$ of $(-\infty, r')$ extending $fgf^{-1}$ (i.e. mapping red 
intervals to red intervals and blue to blue); $h'$ is then extended to the whole of $\mathbb Q$ by letting it fix $[r', \infty)$ pointwise. Then we can see that $fg = h'f$ holds. If $x < q$ then 
$f(x) \in {\rm im} f \cap (-\infty, r')$ so $fg(x) = fgf^{-1}(f(x)) = h'f(x)$, and if $x \ge q$, then $g(x) = x$, so
$fg(x) = f(x) = h'f(x)$ since $f(x) \ge r$ which is fixed by $h'$. Furthermore, the fact that $h'$ has a single non-trivial orbital which is $(-\infty, r')$, on which it has the same parity as $g$,
follows easily. It follows that $M \models {\rm act}_1(f, g, h)$. 

Now we remark on how the argument is modified if $f \in \Gamma^-$ (or similarly the other sets). The choices of $r'$ and $k$ are as before. The partial map $fgf^{-1}$ is again considered on 
${\rm im}(f) \cap (-\infty, r')$, but this time, this set is bounded below. This makes no essential difference, except that $h'$ will also fix points of $(-\infty, \inf {\rm im}(f))$ pointwise.  \end{proof} 

\begin{theorem}\label{3.4} The ordered set $({\mathbb Q}, <)$ and the action of its monoid of embeddings $M$ on $\mathbb Q$ are first order interpretable in $(M, \circ)$.  
\end{theorem}
 
\begin{proof} Again note that since $G$ is equal to the set of invertible elements of $M$, it is a first order definable subset of $M$ in the monoid language, so we may refer to members of $G$ in any 
definitions. Thus by Theorem \ref{2.13} we may represent the points of $\mathbb Q$ using group elements satisfying the formula `rational'. To deduce how $M$ acts from Lemma \ref{3.3}, we need to following 
formula:

\vspace{.1in}

\noindent $\mbox{act}_2(x, y, z) : \exists x_1 \exists x_2 \exists t(\mbox{act}_1(x_1, y, t) \wedge \mbox{act}_1(x_2, z, t) \wedge x_1 = x_2x)$.

\vspace{.1in}

Firstly, if $M \models \mbox{act}_2(f,g,h)$, let $f_1, f_2$, and $k$ be witnesses for $x_1, x_2, t$ respectively. 
Thus $M \models \mbox{act}_1(f_1,g,k)$ and $M \models \mbox{act}_1(f_2,h,k)$. Let $q$, $r$ and $s$ be the rationals encoded by $g$, $h$ and $k$ respectively. By Lemma \ref{3.3}, $f_1(q) = s = f_2(r)$. Also, 
$f_1 = f_2f$. So $f_2f(q) = f_1(q) = f_2(r)$, and as $f_2$ is injective, $f(q) = r$.

Conversely, suppose that $f(q) = r$ and let $g$ and $h$ be conjugate elements satisfying `rational' having supports $(-\infty, q)$ and $(-\infty, r)$ respectively. If ${\rm im}(f)$ is coterminal, then by 
\cite{Edith}, Lemma 2.4, there is $f_2 \in \Gamma$ such that $f_1 = f_2f \in \Gamma$. Let $s = f_1(q)$ and let $k$ satisfy `rational', encode $s$, and be conjugate to $g$. Then 
$s= f_1(q) = f_2f(q) = f_2(r)$. By Lemma \ref{3.3}, $M \models {\rm act}_1(f_1, g, k)$ and $M \models {\rm act}_1(f_2, h, k)$, and so this shows that $M \models {\rm act}_2(f, g, h)$ as desired. If however 
${\rm im}(f)$ is bounded below but not above, then by \cite{Edith}, Lemma 2.4, there is $f_1 \in \Gamma^-$ such that $f_1f \in \Gamma^-$, and the same proof shows that $M \models {\rm act}_2(f, g, h)$, with 
similar modifications in the other cases, ${\rm im}(f)$ bounded above but not below, using $\Gamma^+$, or bounded above and below using $\Gamma^\pm$.          \end{proof}

We now show how to derive the precisely analogous result for the endomorphism monoid $E$ of $({\mathbb Q}, \le)$. We collect together some results we need from \cite{Edith} in the following lemma.

\begin{lemma}\label{3.5} (i) Each $f \in S$ has a right inverse, and any such right inverse lies in $M$.

(ii) For any $h \in E$ there are $f \in M$ and $g \in S$ such that $h = gf$. \end{lemma}

\begin{theorem} \label{3.6} The ordered set $({\mathbb Q}, <)$ and the action of its monoid of endomorphisms $E$ on $\mathbb Q$ are first order interpretable in $(E, \circ)$.   \end{theorem}
 
\begin{proof}  First we may represent the action of $S$ (the epimorphisms) thus:

\vspace{.1in}

\noindent $\mbox{act}_3(x, y, z) : \exists t(xt = 1 \wedge \mbox{act}_2(t, z, y))$.

\vspace{.1in}

We show that for any $f \in S$, and $g, h \in G$, $E \models \mbox{act}_3(f, g, h)$ if and only if $f(q) = r$ where $q$ and $r$ are the rationals encoded by $g$ and $h$ respectively. To see that this is 
correct, first suppose that $E \models \mbox{act}_3(f, g, h)$. Then there is $f'$ such that $ff' = 1$, and by Lemma \ref{3.5}(i), $f' \in M$. Since $\mbox{act}_2(f', h, g)$, from the proof of Theorem 
\ref{3.4} we deduce that $f'(r) = q$, and it follows that $f(q) = r$. Conversely, if $f(q) = r$, there is some right inverse $f'$ of $f$ such that $f'(r) = q$, and $f' \in M$. Therefore 
$M \models \mbox{act}_2(f', h, g)$, and so also $E \models \mbox{act}_3(f, g, h)$.     

Now we move on to $f \in E$, and we consider the following formula:

\vspace{.1in}

\noindent $\mbox{act}_4(x, y, z) : \exists x_1 \exists x_2 \exists t(x = x_1x_2 \wedge \mbox{act}_3(x_1, t, z) \wedge \mbox{act}_2(x_2, y, t))$.

\vspace{.1in}

\noindent We show that $E \models \mbox{act}_4(f, g, h)$ if and only if $g$ and $h$ encode rationals $q$ and $r$ respectively such that $f(q) = r$. 
 
First suppose that $E \models \mbox{act}_4(f, g, h)$, and let $f_1$, $f_2$, and $k$ be witnesses for $x_1$, $x_2$, and $t$ respectively. Thus $f = f_1f_2$ and $E \models \mbox{act}_3(f_1, k, h)$  and
$E \models \mbox{act}_2(f_2, g, k)$. In particular, $E \models \mbox{rational}(k)$, so $k$ encodes some rational number $s$ say. Thus $f_1(s) = r$ and $f_2(q) = s$. Therefore 
$f(q) = f_1f_2(q) = f_1(s) = r$. Conversely, suppose that $f(q) = r$, and that $g$ and $h$ are members of $G$ satisfying `rational' encoding $q$ and $r$ respectively. By Lemma \ref{3.5}(ii) there are
$f_1 \in S$ and $f_2 \in M$ such that $f = f_1f_2$. Let $f_2(q) = s$, and let $k$ be a member of $G$ satisfying `rational' which encodes $s$. Then $f_1(s) = f_1f_2(q) = r$, and therefore
$E \models \mbox{act}_3(f_1, k, h)$ and $E \models \mbox{act}_2(f_2, g, k)$. It follows that $E \models \mbox{act}_4(f, g, h)$.         \end{proof}

\end{document}